\documentclass{amsart}
\usepackage{amssymb}
\usepackage{xcolor}
\title{Extremal Aspects of the Erd\H os--Gallai--Tuza Conjecture}
\author{Gregory J.~Puleo}
\address{Department of Mathematics, University of Illinois at Urbana-Champaign. Now at Coordinated Science Lab, University of Illinois at Urbana-Champaign.}
\renewcommand{\subset}{\subseteq}
\newcommand{\st}{\colon\,}
\newcommand{\aphtau}{f_1}
\newcommand{\aphtaub}{f_{\mathrm{B}}}
\newcommand{\join}{\vee}
\newcommand{\iso}{\cong}
\newcommand{\lfrac}[2]{#1/#2}

\newcommand{\bad}[1]{\textcolor{red}{\emph{#1}}}
\newcommand{\citethis}[1]{\bad{[CITE THIS]}}
\newcommand{\caze}[2]{\textbf{Case {#1}:} \textit{#2}}
\newcommand{\sizeof}[1]{\left\lvert{#1}\right\rvert}

\newcommand{\aph}{\alpha_1}

\let\oldtau\tau
\renewcommand{\tau}{\oldtau_1}
\newcommand{\sbar}{\overline{S}}
\newcommand{\taub}{\oldtau_{\mathrm{B}}}
\newtheorem{proposition}{Proposition}[section]
\newtheorem{conjecture}[proposition]{Conjecture}
\newtheorem{lemma}[proposition]{Lemma}
\newtheorem{theorem}[proposition]{Theorem}
\newtheorem{corollary}[proposition]{Corollary}
\theoremstyle{definition}

\theoremstyle{remark}

\begin{document}
\begin{abstract}
  Erd\H os, Gallai, and Tuza posed the following problem: given an
  $n$-vertex graph $G$, let $\tau(G)$ denote the smallest size of a
  set of edges whose deletion makes $G$ triangle-free, and let
  $\aph(G)$ denote the largest size of a set of edges containing at
  most one edge from each triangle of $G$. Is it always the case that
  $\aph(G) + \tau(G) \leq n^2/4$?  We also consider a variant on this
  conjecture: if $\taub(G)$ is the smallest size of an edge set whose
  deletion makes $G$ bipartite, does the stronger inequality $\aph(G)
  + \taub(G) \leq n^2/4$ always hold?
  
  By considering the structure of a minimal counterexample to each
  version of the conjecture, we obtain two main results. Our first
  result states that any minimum counterexample to the original Erd\H
  os--Gallai--Tuza Conjecture has ``dense edge cuts'', and in
  particular has minimum degree greater than $n/2$.  This implies that
  the conjecture holds for all graphs if and only if it holds for all
  triangular graphs (graphs where every edge lies in a triangle).  Our
  second result states that $\aph(G) + \taub(G) \leq n^2/4$ whenever
  $G$ has no induced subgraph isomorphic to $K_4^-$, the graph
  obtained from the complete graph $K_4$ by deleting an edge.  Thus,
  the original conjecture also holds for such graphs.

\end{abstract}
\maketitle
\section{Introduction}
Given an $n$-vertex graph $G$, say that a set $A \subset E(G)$ is
\emph{triangle-independent} if it contains at most one edge from each
triangle of $G$, and say that $X \subset E(G)$ is a \emph{triangle
  edge cover} if $G \setminus X$ is triangle-free. Throughout this paper,
$\aph(G)$ denotes the maximum size of a triangle-independent set
of edges in $G$, while $\tau(G)$ denotes the minimum size of a triangle
edge cover in $G$.

Erd\H os~\cite{largebip} showed that every $n$-vertex graph $G$ has a
bipartite subgraph with at least $\sizeof{E(G)}/2$ edges, which implies that
$\tau(G) \leq \sizeof{E(G)}/2 \leq n^2/4$. Similarly, if $A$ is
triangle-independent, then the subgraph of $G$ with edge set $A$ is
clearly triangle-free; by Mantel's~Theorem, this implies that $\aph(G)
\leq n^2/4$. The Erd\H os--Gallai--Tuza conjecture is a common generalization
of these upper bounds.
\begin{conjecture}[Erd\H os--Gallai--Tuza \cite{EGT}]\label{coj:EGT}
  For every $n$-vertex graph $G$, $\aph(G) + \tau(G) \leq n^2/4$.
\end{conjecture}
The conjecture is sharp, if true: consider the graphs $K_{n}$ and
$K_{n/2, n/2}$, where $n$ is even. We have $\aph(K_n) = n/2$ and
$\tau(K_n) = {n \choose 2} - n^2/4$, while $\aph(K_{n/2, n/2}) =
n^2/4$ and $\tau(K_{n/2, n/2}) = 0$.  In both cases, $\aph(G) +
\tau(G) = n^2/4$, but a different term dominates in each case.
More generally, the conjecture is sharp for any graph of the form
$K_{r_1, r_1} \join \cdots \join K_{r_t, r_t}$, a fact which follows from
the characterization of such graphs in \cite{Puleo} as the graphs 
achieving equality in the bound $\aph(G) \leq \frac{n^2}{2} - \sizeof{E(G)}$.

The original paper of Erd\H os, Gallai, and Tuza~\cite{EGT} considered
the conjecture only for \emph{triangular graphs}, which are graphs
such that every edge lies in a triangle. This version of the
conjecture was also stated by Erd\H os as Problem~17.2 of
\cite{OldNew}. Later formulations of the conjecture, such as
\cite{ErdosSelection} and \cite{tuza2001unsolved}, dropped the
triangularity requirement, and instead stated the conjecture for
general graphs; this discrepancy was pointed out by Grinberg on
MathOverflow~\cite{Grinberg}, who asked if the two formulations were
really equivalent. Our results in this paper imply that the two forms
of the conjecture are equivalent, settling Grinberg's question.

Throughout the paper, we use the term \emph{minimal counterexample} to
refer to a vertex-minimal counterexample, that is, a graph $G$ such
that the property in question holds for every proper induced subgraph
of $G$. When $S \subset V(G)$, we write $\sbar$ for the set $V(G) -
S$, and we write $[S, \sbar]$ for the \emph{edge cut} between $S$ and
$\sbar$, that is, the set of all edges with one endpoint in $S$ and
the other endpoint in $\sbar$.

In Section~\ref{sec:plaintau}, we prove that if $G$ is a minimal
counterexample to Conjecture~\ref{coj:EGT}, then for every nonempty proper vertex subset
$S$, the edge cut $[S, \sbar]$ has more than
$\sizeof{S}(n-\sizeof{S})/2$ edges. A small refinement of the argument
shows that $\delta(G) > n/2$ whenever $G$ is a minimal
counterexample. Thus, any minimal counterexample is a triangular
graph, so if Conjecture~\ref{coj:EGT} holds for triangular graphs,
then no counterexample exists.

We then consider the following stronger variant on
Conjecture~\ref{coj:EGT}, proposed by Lehel (see \cite{OldNew}) and
independently proposed by the author~\cite{Puleo}. Let $\taub(G)$
denote the smallest size of an edge set $X$ such that $G-X$ is
bipartite; note that $\taub(G) \geq \tau(G)$.
\begin{conjecture}\label{coj:puleo}
  For every $n$-vertex graph $G$, $\aph(G) + \taub(G) \leq n^2/4$.  
\end{conjecture}
A partial result~\cite{Puleo} towards Conjecture~\ref{coj:puleo}, and
thus towards Conjecture~\ref{coj:EGT}, states that
$\aph(G) + \taub(G) \leq 5n^2/16$ for every graph $G$. In
Section~\ref{sec:biptau}, we study the properties of a minimal
counterexample to Conjecture~\ref{coj:puleo}, obtaining a ``dense
cuts'' result similar to that of Section~\ref{sec:plaintau} (but
somewhat more complicated to state). This theorem implies that if $G$
has no induced subgraph isomorphic to $K_4^-$, then
$\aph(G) + \taub(G) \leq n^2/4$.  Although this class of graphs is
highly constrained, it includes $K_n$ and $K_{n/2,n/2}$, the two extremes
of the family of motivating sharpness examples.

\section{Dense Cuts in a Minimal Counterexample}\label{sec:plaintau}
Erd\H os, Gallai, and Tuza~\cite{EGT} showed that $\aph(G) + \tau(G)
\leq \sizeof{E(G)}$ for all $G$, via the following argument: if $A
\subset E(G)$ is triangle-independent, then $E(G) - A$ contains at
least $2$ edges from each triangle of $G$, so $E(G) - A$ is a triangle
edge cover. This argument is ``global'', dealing with all edges in $G$;
we ``localize'' it, dealing only with edges in some edge cut
$[S, \sbar]$ for $S \subset V(G)$.

To avoid clutter, we write $\aphtau(G)$ for the sum $\aph(G) + \tau(G)$.
\begin{lemma}\label{lem:peel}
  If $S$ is nonempty proper subset of $V(G)$, then
  \[ \aphtau(G) \leq \aphtau(G[S]) + \aphtau(G[\sbar]) + \sizeof{[S, \sbar]}. \]
\end{lemma}
\begin{proof}
  Let $A \subset E(G)$ be a largest triangle-independent set in $G$,
  let $G_1 = G[S]$, and let $G_2 = G[\sbar]$. For $i \in
  \{1,2\}$, let $A_i = A \cap E(G_i)$, so that $A_i$ is a
  triangle-independent set in $G_i$, and let $B = A \cap [S,
  \sbar]$. Since $\sizeof{A_i}$ is a lower bound on
  $\aph(G_i)$, we have
  \[ \aph(G) = \sizeof{A} = \sizeof{A_1} + \sizeof{A_2} + \sizeof{B}
  \leq \aph(G_1) + \aph(G_2) + \sizeof{B}. \]
  Next, let $X_i$ be a minimum triangle edge cover in $G_i$ for $i \in
  \{1,2\}$, so that $\sizeof{X_i} = \tau(G_i)$, and let $Y = [S,
  \sbar] - B$. We claim that $X_1 \cup X_2 \cup Y$ is
  a triangle edge cover in $G$. Clearly $X_i$ covers all triangles
  contained in $V(G_i)$, so it suffices to show that $Y$ covers
  all triangles intersecting both $S$ and $\sbar$. If
  $T$ is such a triangle, then two edges of $T$ lie in $[S, \sbar]$.
  Since $B \subset A$ and $A$ is triangle-independent, at most one of these edges is
  contained in $B$; the other lies in $Y$. Hence $X_1 \cup X_2 \cup Y$
  is a triangle edge cover in $G$, and we conclude that  
  \[ \tau(G) \leq \sizeof{X_1} + \sizeof{X_2} + \sizeof{Y}
  = \tau(G_1) + \tau(G_2) + \big(\sizeof{[S, \sbar]} - \sizeof{B}\big). \]
  Combining the bounds on $\aph(G)$ and $\tau(G)$ yields the desired inequality.
\end{proof}
\begin{theorem}\label{thm:ctxcut}
  Let $G$ be a minimal counterexample to Conjecture~\ref{coj:EGT}. If
  $S$ is a nonempty proper subset of $V(G)$, then $\sizeof{[S,
    \sbar]} > \frac{1}{2}\sizeof{S}(n - \sizeof{S})$, where $n
  = \sizeof{V(G)}$.
\end{theorem}
\begin{proof}
  Let $G_1 = G[S]$ and let $G_2 = G[\sbar]$. Since $G$ is a
  minimal counterexample, we have
  \begin{align*}
    \aph(G_1) + \tau(G_1) &\leq \sizeof{S}^2/4, \\
    \aph(G_2) + \tau(G_2) &\leq (n - \sizeof{S})^2/4.
  \end{align*}
  By Lemma~\ref{lem:peel}, it follows that
  \[ \aph(G) + \tau(G) \leq \frac{n^2}{4} - \frac{\sizeof{S}(n-\sizeof{S})}{2} + \sizeof{[S, \sbar]}. \]
  Since $\aph(G) + \tau(G) > n^2/4$, the claim follows.
\end{proof}
Applying Theorem~\ref{thm:ctxcut} to a set consisting precisely of a
vertex of minimum degree yields the lower bound $\delta(G) >
(n-1)/2$.
Parity considerations allow us to obtain the stronger bound
$\delta(G) > n/2$.
\begin{theorem}\label{lem:mindeg}
  If $G$ is a minimal counterexample to Conjecture~\ref{coj:EGT},
  then $\delta(G) > n/2$, where $n = \sizeof{V(G)}$.
\end{theorem}
\begin{proof}
  Let $v$ be a vertex of minimum degree in $G$, and let $G_0 = G-v$.
  By minimality, $\aph(G_0) + \tau(G_0) \leq (n-1)^2/4$. By
  Lemma~\ref{lem:peel}, since $G$ is a counterexample we have
  \begin{equation}
    \label{eqn:cat}
    \frac{n^2}{4} < \aph(G_0) + \tau(G_0) + d(v).
  \end{equation}
  We split into cases according to the parity of $n$.

  \caze{1}{$n$ is odd.} Using Inequality~\eqref{eqn:cat}, we have
  \[ \frac{2n - 1}{4} = \frac{n^2 - (n-1)^2}{4} < d(v), \]
  so $d(v) > \frac{n}{2} - \frac{1}{4}$, which implies $d(v) > n/2$
  since $n$ is odd.

  \caze{2}{$n$ is even.}
  Since $\aph(G_0) + \tau(G_0)$ is
  an integer, the condition $\aph(G_0) + \tau(G_0) \leq (n-1)^2/4$
  implies
  \[ \aph(G_0) + \tau(G_0) \leq \frac{n^2 - 2n}{4} = \frac{n^2}{4} - \frac{n}{2}. \]
  Therefore, Inequality~\eqref{eqn:cat} implies
  \[ \frac{n^2}{4} < \frac{n^2}{4} - \frac{n}{2} + d(v), \]
  which again easily yields $d(v) > n/2$.
\end{proof}
Since $\delta(G) > n/2$ implies that $G$ is triangular, we have the following
corollary.
\begin{corollary}\label{cor:triang}
  If Conjecture~\ref{coj:EGT} holds for all triangular graphs, then Conjecture~\ref{coj:EGT}
  holds for all graphs.
\end{corollary}
\section{Dense Cuts in the $\taub$ Variant}\label{sec:biptau}
In this section, we consider Conjecture~\ref{coj:puleo}, which deals
with the sum $\aph(G) + \taub(G)$. We again focus on edge cuts in a
minimal counterexample to the conjecture $\aph(G) + \taub(G) \leq
n^2/4$. The development is analogous to Section~\ref{sec:plaintau},
with some differences.

For shorthand, let $\aphtaub(G) = \aph(G) + \taub(G)$.

\begin{lemma}\label{lem:denseboth}
  Let $G$ be a graph, and let $A$ be a triangle-independent set of edges in $G$.
  If $S$ is a nonempty proper subset of $V(G)$, then  
  \[
  \aphtaub(G) \leq \aphtaub(G[S]) + \aphtaub(G[\sbar]) + \frac{1}{2}\sizeof{[S, \sbar]} + \sizeof{[S, \sbar] \cap A}.    
  \]    
\end{lemma}
\begin{proof}
  Clearly, $\aph(G) \leq \aph(G[S]) + \aph(G[\sbar]) +
  \sizeof{[S, \sbar] \cap A}$, since $A \cap G[S]$ and $A \cap
  G[\sbar]$ are triangle-independent sets in $G[S]$ and
  $G[\sbar]$ respectively.  The bound $\taub(G) \leq
  \taub(G[S]) + \taub(G[\sbar]) + \frac{1}{2}\sizeof{[S,
    \sbar]}$ follows by considering the two different ways to
  join the partite sets of a largest bipartite subgraph in $G[S]$ with
  those of one in $G[\sbar]$.
\end{proof}
Since the conclusion of Lemma~\ref{lem:denseboth} deals with both the
graph $G$ and a triangle-independent set $A$, it is difficult to draw
blanket conclusions about the structure of a minimal counterexample
$G$. However, we can draw some conclusions if we impose restrictions
on the structure of $G[S]$.
\begin{lemma}\label{lem:densemin}
  Let $G$ be a minimal counterexample to Conjecture~\ref{coj:puleo},
  and let $S$ be a nonempty proper subset of $V(G)$. If $G[S]$ has independence number $t$,
  then
  \[ \sizeof{[S, \sbar]} > (\sizeof{S} - 2t)(n - \sizeof{S}), \]  
  where $n = \sizeof{V(G)}$.
\end{lemma}
\begin{proof}
  Let $A$ be any triangle-independent set of edges in $G$, and for
  every $v \in V(G)$, let $N_A(v) = \{w \in V(G) \st vw \in
  A\}$.
  Since $A$ is triangle-independent, the set $N_A(v)$ is independent
  in $G$ for every vertex $v$. Thus, $\sizeof{N_A(v) \cap S} \leq t$
  for each $v \in \sbar$, which yields
  $\sizeof{[S, \sbar] \cap A} \leq t(n-\sizeof{S})$.

  Since $G$ is a minimal graph satisfying $\aph(G) + \taub(G) >
  \sizeof{V(G)}^2/4$, Lemma~\ref{lem:denseboth} implies that
  \[ \frac{1}{2}\sizeof{[S, \sbar]} + \sizeof{[S, \sbar] \cap A} > \frac{\sizeof{S}(n - \sizeof{S})}{2.} \]
  Hence
  \[ \sizeof{[S, \sbar]} > \sizeof{S}(n - \sizeof{S}) - 2\sizeof{[S, \sbar] \cap A} \geq (\sizeof{S}-2t)(n - \sizeof{S}). \]
\end{proof}
Lemma~\ref{lem:densemin} allows us to prove a general result
about maximal cliques in minimal counterexamples.
\begin{lemma}\label{lem:extendclique}
  If $G$ is a minimal counterexample to Conjecture~\ref{coj:puleo}
  and $S$ is a maximal clique in $G$, then $S$ is contained in an
  induced copy of $K_{\sizeof{S}+1}^-$.
\end{lemma}
\begin{proof}
  Since Conjecture~\ref{coj:EGT} is known to hold for complete graphs,
  $S \neq V(G)$. By Lemma~\ref{lem:densemin}, we have
  $\sizeof{[S, \sbar]} > (\sizeof{S} - 2)(n-\sizeof{S})$.  Thus,
  $\sbar$ contains a vertex $v$ such that
  $\sizeof{N(v) \cap S} \geq \sizeof{S} - 1$.  Since $S$ is maximal,
  $G[S \cup \{v\}] \iso K_{\sizeof{S}+1}^-$.
\end{proof}
Before proving that Conjecture~\ref{coj:puleo} holds for graphs
with no induced copy of $K_4^-$, we prove that it holds for all
triangle-free graphs. (This is not trivial, since it is possible
that $\taub(G) > 0$ even though $G$ is triangle-free.)
\begin{lemma}\label{lem:trifree}
  If $G$ is a triangle-free $n$-vertex graph, then $\aph(G) + \taub(G) \leq n^2/4$.
\end{lemma}
\begin{proof}
  This follows from a result of Erd\H os, Faudree, Pach, and
  Spencer~\cite{howbip}, who showed that if $G$ is an $n$-vertex
  triangle-free graph with $m$ edges, then $\taub(G) \leq m - \lfrac{4m^2}{n^2}$.
  This inequality immediately implies $\aph(G) + \taub(G) \leq 2m - \lfrac{4m^2}{n^2}$;
  maximizing the upper bound over $m$ yields $\aph(G) + \taub(G) \leq n^2/4$.
\end{proof}
\begin{theorem}
  If $G$ is an $n$-vertex graph with no induced copy of $K_4^-$, then
  $\aph(G) + \taub(G) \leq n^2/4$. 
\end{theorem}
\begin{proof}
  If not, let $G$ be a minimal graph with no induced copy of $K_4^-$
  for which $\aphtaub(G) > n^2/4$. If $G'$ is an induced subgraph of
  $G$ with $n'$ vertices, then $G'$ also has no induced copy of
  $K_4^-$, so $\aphtaub(G') \leq (n')^2/4$.  Thus, $G$ is a minimal
  counterexample to Conjecture~\ref{coj:puleo}.

  Let $S$ be a clique of maximum size in $G$. By Lemma~\ref{lem:trifree}, we have
  $\sizeof{S} \geq 3$. By Lemma~\ref{lem:extendclique}, it follows
  that $S$ is contained in some induced copy of
  $K_{\sizeof{S}+1}^{-}$, which contradicts the assumption that $G$
  has no induced copy of $K_4^-$.
\end{proof}
\section{Acknowledgments}
The author acknowledges support from the IC Postdoctoral Research Fellowship.
\bibliographystyle{amsplain}\bibliography{sumbib}

\providecommand{\bysame}{\leavevmode\hbox to3em{\hrulefill}\thinspace}
\providecommand{\MR}{\relax\ifhmode\unskip\space\fi MR }
\providecommand{\MRhref}[2]{%
  \href{http://www.ams.org/mathscinet-getitem?mr=#1}{#2}
}
\providecommand{\href}[2]{#2}
\begin{thebibliography}{1}

\bibitem{largebip}
Paul Erd{\H{o}}s, \emph{On some extremal problems in graph theory}, Israel J.
  Math. \textbf{3} (1965), 113--116. \MR{0190027 (32 \#7443)}

\bibitem{OldNew}
\bysame, \emph{Some of my old and new combinatorial problems}, Paths, flows,
  and {VLSI}-layout ({B}onn, 1988), Algorithms Combin., vol.~9, Springer,
  Berlin, 1990, pp.~35--45. \MR{1083376 (91j:05001)}

\bibitem{ErdosSelection}
\bysame, \emph{A selection of problems and results in combinatorics}, Combin.
  Probab. Comput. \textbf{8} (1999), no.~1-2, 1--6, Recent trends in
  combinatorics (M{\'a}trah{\'a}za, 1995). \MR{1684620 (2000b:05002)}

\bibitem{howbip}
Paul Erd{\H{o}}s, Ralph Faudree, J{\'a}nos Pach, and Joel Spencer, \emph{How to
  make a graph bipartite}, J. Combin. Theory Ser. B \textbf{45} (1988), no.~1,
  86--98. \MR{953897 (89f:05134)}

\bibitem{EGT}
Paul Erd{\H{o}}s, Tibor Gallai, and Zsolt Tuza, \emph{Covering and independence
  in triangle structures}, Discrete Math. \textbf{150} (1996), no.~1-3,
  89--101, Selected papers in honour of Paul Erd{\H{o}}s on the occasion of his
  80th birthday (Keszthely, 1993). \MR{1392722 (97d:05222)}

\bibitem{Grinberg}
Darij Grinberg, \emph{What is the correct statement of this
  {E}rd{\H{o}}s-{G}allai-{T}uza problem generalizing {T}uran's triangle
  theorem?}, MathOverflow, URL:http://mathoverflow.net/q/87484 (version:
  2012-02-03).

\bibitem{Puleo}
Gregory~J. Puleo, \emph{On a conjecture of {E}rd{\H{o}}s, {G}allai, and
  {T}uza}, Journal of Graph Theory (2014+), to appear. arXiv preprint
  arXiv:1311.5332, doi 10.1002/jgt.21828.

\bibitem{tuza2001unsolved}
Zsolt Tuza, \emph{Unsolved combinatorial problems part {I}}, BRICS Lecture
  Series (2001), no.~LS-01-1, 30--30.

\end{thebibliography}
\end{document}